\documentclass[12pt]{amsart}

\newtheorem{theorem}{Theorem}[section]
\newtheorem{proposition}[theorem]{Proposition}
\newtheorem{lemma}[theorem]{Lemma}

\theoremstyle{remark}
\newtheorem{remark}[theorem]{Remark}

\numberwithin{equation}{section}

\newcommand{\gm}{{\mathbb G}_m}

\newcommand{\sQ}{\mathcal{Q}}

\newcommand{\sO}{\mathcal{O}}

\newcommand{\sE}{\mathcal{E}}
\newcommand{\sF}{{\mathcal{F}}}

\newcommand{\sK}{{\mathcal{K}}}
\newcommand{\sL}{{\mathcal{L}}}

\newcommand{\bP}{{\bf P}}
\newcommand{\bQ}{{\bf Q}}

\newcommand{\rQ}{{\rm Q}}

\newcommand{\univ}{{\rm univ}}
\DeclareMathOperator{\quot}{Quot}
\DeclareMathOperator{\sym}{Sym}

\DeclareMathOperator{\Hom}{Hom}

\begin{document}

\title[A generalized Quot scheme and meromorphic vortices]{A generalized Quot scheme
and meromorphic vortices}

\author[I. Biswas]{Indranil Biswas}

\address{School of Mathematics, Tata Institute of Fundamental Research,
Homi Bhabha Road, Bombay 400005, India}

\email{indranil@math.tifr.res.in}

\author[A. Dhillon]{Ajneet Dhillon}

\address{Department of Mathematics, Middlesex College, University of
Western Ontario, London, ON N6A 5B7, Canada}

\email{adhill3@uwo.ca}

\author[J. Hurtubise]{Jacques Hurtubise}

\address{Department of Mathematics, McGill University, Burnside
Hall, 805 Sherbrooke St. W., Montreal, Que. H3A 0B9, Canada}

\email{jacques.hurtubise@mcgill.ca}

\author[R. A. Wentworth]{Richard A. Wentworth}

\address{Department of Mathematics, University of Maryland, College Park,
MD 20742, USA}

\email{raw@umd.edu}



\date{}

\begin{abstract}
Let $X$ be a compact connected Riemann surface. Fix a positive integer
$r$ and two nonnegative integers $d_p$ and $d_z$. Consider all pairs of the form
$({\mathcal F}\, ,f)$, where $\mathcal F$ is a holomorphic vector bundle on $X$
of rank $r$ and degree $d_z-d_p$, and
$$
f\, :\, {\mathcal O}^{\oplus r}_X\, \longrightarrow\, {\mathcal F}
$$
is a meromorphic homomorphism which an isomorphism outside a finite subset of $X$ and has
pole (respectively, zero) of total degree $d_p$ (respectively, $d_z$). Two such pairs
$({\mathcal F}_1\, ,f_1)$ and $({\mathcal F}_2\, ,f_2)$ are called isomorphic if
there is a holomorphic
isomorphism of ${\mathcal F}_1$ with ${\mathcal F}_2$ over $X$ that takes $f_1$
to $f_2$. We construct a natural compactification of the moduli space equivalence classes 
pairs of the above type. The Poincar\'e polynomial of this compactification is computed.
\end{abstract}

\maketitle

\section{Introduction}\label{sec1}

Take a compact connected Riemann surface $X$. Fix positive integers $r$ and $d$.
Consider pairs of the form $(E\, ,f)$, where $E$ is a holomorphic vector bundle on
$X$ of rank $r$ and degree $d$, and
$$
f\,:\, {\mathcal O}^{\oplus r}_X\, \longrightarrow\, E
$$
is an ${\mathcal O}_X$--linear homomorphism which is an isomorphism outside a finite
subset of $X$. This implies that the total degree of zeros of $f$ is $d$. Two such
pairs $(E_1\, ,f_1)$ and $(E_2\, ,f_2)$ are called equivalent if there is a
holomorphic isomorphism
$$
\phi\, :\, E_1\, \longrightarrow\, E_2
$$
such that $\phi\circ f_1\,=\, f_2$. Such pairs are examples of vortices
\cite{BDW}, \cite{Br}, \cite{BR}, \cite{Ba}, \cite{EINOS}. 

For any pair $(E\, ,f)$ of the above type, consider the dual homomorphism
$$
f^*\,:\, E^*\,\longrightarrow\, ({\mathcal O}^{\oplus r}_X)^*\,=\,
{\mathcal O}^{\oplus r}_X\, .
$$
The quotient ${\mathcal O}^{\oplus r}_X/\text{image}(f^*)$ is an element of the
Quot scheme $\text{Quot}(r,d)$ that parametrizes all torsion quotients of
${\mathcal O}^{\oplus r}_X$ of degree $d$. Conversely, given any torsion quotient
$$
{\mathcal O}^{\oplus r}_X\, \stackrel{\psi}{\longrightarrow}\, T
$$
of degree $d$, consider the homomorphism
$$
{\mathcal O}^{\oplus r}_X\,=\, ({\mathcal O}^{\oplus r}_X)^*\,
\stackrel{\psi'}{\longrightarrow}\, \text{kernel}(\psi)^*
$$
induced by the inclusion $\text{kernel}(\psi)\,\hookrightarrow\,
{\mathcal O}^{\oplus r}_X$. The pair $(\text{kernel}(\psi)^*\, , \psi')$ is clearly
of the above type. Therefore, the moduli space of
equivalence classes of pairs $(E\, ,f)$ is identified with the Quot scheme
$\text{Quot}(r,d)$. 

Here we consider pairs of the form $(E\, ,f)$, where $E$ is a holomorphic vector
bundle on $X$ of rank $r$ and degree $d$, and
$$
f\,:\, {\mathcal O}^{\oplus r}_X\, \longrightarrow\, E
$$
is an ${\mathcal O}_X$--linear meromorphic homomorphism which is an isomorphism
outside a finite subset of $X$. We assume that the total degree of the poles of the
meromorphic homomorphism is $d_p$. This implies that the total degree of the zeros
of the meromorphic homomorphism is $d+d_p$. As before, two such pairs
$(E_1\, ,f_1)$ and $(E_2\, ,f_2)$ will be called equivalent if there is a holomorphic
isomorphism
$$
\phi\, :\, E_1\, \longrightarrow\, E_2
$$
such that $\phi\circ f_1\,=\, f_2$. The equivalence classes of pairs can be considered
as examples of meromorphic vortices.

We construct a natural compactification of the moduli space of these
meromorphic vortices. We compute the Poincar\'e polynomial of this compactification.

\section{Preliminaries}

Let $S$ be a scheme and $Y\,\longrightarrow\, S$ a smooth projective morphism. Given
a coherent sheaf $\sF$ on $Y$ flat over $S$ and a numerical polynomial $r(t)$, we denote
by $\quot(\sF/S, r(t))$ the Grothendieck Quot scheme over $S$ parametrizing quotients
of $\sF$ with Hilbert polynomial $r(t)$ \cite{Gr}. There is a universal exact sequence on
$\quot(\sF/S,r(t))\times_S Y$
\[
0\,\longrightarrow\, \sK^\univ_{\quot(\sF/S,r(t))}\,\longrightarrow\, \pi^*_{Y}\sF
\,\longrightarrow\,\sQ^\univ_{\quot(\sF/S,r(t))}\,\longrightarrow\, 0\, ,
\]
where $\pi_Y\,:\, \quot(\sF/S,r(t))\times_S Y\,\longrightarrow\, Y$ is the natural
projection. Often we will just drop the subscripts and write $\sK^\univ$ or ${\sQ}^\univ$
instead. This construction is well behaved with respect to pull-backs, so let us record
the following:

\begin{lemma}\label{l:qPullback}
For any morphism $g\,:\,T\,\longrightarrow\, S$, the base change
\[
\quot(g^*\sF/T,r(t))\,\cong \, \quot(\sF/S,r(t))\times_S T
\]
holds.
\end{lemma}

\begin{proof}
This follows by examining the corresponding representable functors.
\end{proof}

We will mostly be interested in the case where $Y\,\longrightarrow\, S$ is a smooth,
connected and of relative dimension one, that is a relative curve, and $\sF$ is
locally free of rank $r$. Further, we will only consider torsion quotients
of rank zero and degree $d$. This  Quot scheme will be denoted by 
$\quot(\sF/S,d)$. When $r\,=\,1$ and $S$ is a point, then
\[
\quot(\sO,d)\,=\,\sym^d (Y)\, ,
\]
the $d$-th symmetric power of $Y$.

Given an positive integer $d$ by \emph{a partition of length $k>0$ of $d$} we mean
a sequence $\bP\,=\,(p_1\, ,p_2\, ,\cdots\, ,p_k)$ of non-negative integers with
$\sum_{i=1}^k p_i\,=\, d$. For such a partition define $d(\bP)\,:=\,\sum_{i=1}^k
(i-1)p_i$. We will write
\[
\sym^\bP (Y) \,=\, \sym^{p_1}(Y) \times \cdots \times \sym^{p_r}(Y)\, .
\]

\section{A relative Quot scheme}

Let $X$ be a compact connected Riemann surface. Let $\sE$ and $\sF$ be two
holomorphic vector bundles on $X$ of common rank $r$. Take a dense open subset $U\,
\subset\, X$, such that the complement $S\, :=\, X\setminus U$ is a finite set, and
take an isomorphism of coherent analytic sheaves
$$
f\, :\, \sE\vert_U\, \longrightarrow\, \sF\vert_U
$$
over $U$. This homomorphism $f$ will be called \textit{meromorphic} if there is
a positive integer $n$ such that $f$ extends to a homomorphism of coherent analytic sheaves
$$
\widehat{f} \,:\, \sE\, \longrightarrow\, \sF\otimes {\mathcal O}_X(nS)\, \supset\, \sF
$$
over $X$, where $S$ is the reduced divisor defined by the finite subset $S$.
Note that since the divisor $S$ is effective, we have $\sF\,\subset\, \sF\otimes
{\mathcal O}_X(nS)$. Therefore, $f$ is meromorphic if and only if the homomorphism $f$
is algebraic with respect to the algebraic structures on $\sE\vert_U$ and $\sF\vert_U$
given by the algebraic structures on $\sE$ and $\sF$ respectively.

Take a meromorphic homomorphism $f$ as above. We note that the extension $\widehat{f}$
is uniquely determined by $f$ because $f$ and $\widehat{f}$ coincide over $U$. The
inverse image
$$
\sE(f)\, :=\, \widehat{f}^{-1}(\sF)\, \subset\, \sE
$$
(recall that $\sF\,\subset\, \sF\otimes{\mathcal O}_X(nS)$)
is clearly independent of the choice of $n$. We note that both $\sE(f)$ and $\widehat{f}
(\sE(f))$ are holomorphic vector bundles on $X$ because they are coherent analytic
subsheaves of holomorphic vector bundles. Both of then are of rank $r$, and the restriction
\begin{equation}\label{e0}
\widehat{f}\vert_{\sE(f)}\, :\, \sE(f) \, \longrightarrow\, \widehat{f}(\sE(f))
\end{equation}
is an isomorphism of holomorphic vector bundles. Define
\begin{equation}\label{e1}
\sQ_p(f)\, :=\, \sE/\sE(f)\ \ ~ \text{ and }~ \ \ \sQ_z(f)\,:=\, \sF/(\widehat{f}(\sE(f)))
\end{equation}
(the subscripts ``$p$'' and ``$z$'' stand for ``pole'' and ``zero'' respectively). We note
that both $\sQ_p(f)$ and $\sQ_z(f)$ are torsion coherent analytic sheaves on $X$. In particular,
their supports are finite subsets of $X$. From \eqref{e1} it follows that
\begin{equation}\label{e2}
\text{degree}(\sQ_p(f))\,=\, \text{degree}(\sE)-\text{degree}(\sE(f))\ \ ~ \text{ and}
\end{equation}
$$
\text{degree}(\sQ_z(f))\,=\, \text{degree}(\sF)- \text{degree}(\widehat{f}(\sE(f)))\, .
$$

Fix positive integers $r$, $d_p$ and $d_z$. Set the domain $\sE$ to be the trivial
vector bundle ${\mathcal O}^{\oplus r}_X$ of rank $r$. Consider all triples of the
form $(\sF\, ,U\, , f)$, where
\begin{itemize}
\item $\sF$ is a holomorphic vector bundle on $X$ of rank $r$,

\item $U$ is the complement of a finite subset of $X$, and

\item $f\, :\, {\mathcal O}^{\oplus r}_X\vert_U\,=\,
{\mathcal O}^{\oplus r}_U\, \longrightarrow\, \sF\vert_U$
is a meromorphic homomorphism such that
$$
\text{degree}(\sQ_p(f))\,=\, d_p \ \ ~ \text{ and } ~ \ \ \text{degree}(\sQ_z(f))\,=\, d_z\, .
$$
\end{itemize}
Since $\widehat{f}\vert_{\sE(f)}$ in \eqref{e0} is an isomorphism, from \eqref{e2} we
conclude that
\begin{equation}\label{e4}
\text{degree}(\sF)\,=\, d_z-d_p+ \text{degree}({\mathcal O}^{\oplus r}_X)\,=\,
d_z-d_p\, .
\end{equation}
Two such triples $(\sF_1\, ,U_1\, ,f_1)$ and $(\sF_2\, ,U_2\, , f_2)$ will be called
\textit{equivalent} if there is a holomorphic isomorphism of vector bundles over $X$
$$
\beta\, :\, \sF_1\, \longrightarrow\, \sF_2
$$
such that
$$\beta\circ (f_1\vert_{U_1\cap U_2})\,=\, f_2\vert_{U_1\cap U_2}\, .
$$
Therefore, the equivalence class of $(\sF\, ,U\, ,f)$ depends only on
$(\sF\, ,f)$ and it is independent of $U$. More precisely, $(\sF\, ,U\, ,f)$
is equivalent to $(\sF\, ,W\, ,f\vert_W)$ for every $W\, \subset\, U$ such that
the complement $U\setminus W$ is a finite set.

Let
\begin{equation}\label{q0}
{\rQ}^0\, =\, {\rQ}^0_X(r, d_p,d_z)
\end{equation}
be the space of all equivalence classes of triples of the above form. We will embed
${\rQ}^0$ as a Zariski open subset of a smooth complex projective variety.

Take any triple $(\sF\, ,U\, , f)$ as above that is represented by a point of ${\rQ}^0$.
Consider the short exact sequence
\begin{equation}\label{e3}
0\,\longrightarrow\, \sE(f)\, :=\, \text{kernel}(q_p)\, \longrightarrow\, \sE\,=\,
{\mathcal O}^{\oplus r}_X\, \stackrel{q_p}{\longrightarrow}\, \sQ_p(f)\,\longrightarrow\, 0\, ,
\end{equation}
where $q_p$ denotes the projection to the quotient in \eqref{e1}. We also have
$$
\sE(f)\,=\, \widehat{f}(\sE(f))\, \hookrightarrow\, \sF
$$
(recall that $\widehat{f}\vert_{\sE(f)}$ in \eqref{e0} is an isomorphism). Let
\begin{equation}\label{e-1}
0\, \longrightarrow\, \sF^*\, \longrightarrow\, \sE(f)^*
\end{equation}
be the dual of the above inclusion of $\sE(f)$ in $\sF$. From \eqref{e3} we have
$\text{degree}(\sE(f)^*)\,=\, \text{degree}(\sQ_p(f))\,=\, d_p$.
Therefore, from \eqref{e4} it follows that
$$
\text{degree}(\sE(f)^*/\sF^*)\,=\, \text{degree}(\sE(f)^*)-\text{degree}(\sF^*)\,=\,
d_p+ d_z-d_p\,=\, d_z
$$
as $\text{degree}(\sF^*)\,=\, -\text{degree}(\sF)$.
These imply that we can recover the equivalence class of $(\sF\, ,f)$ once we know
the following two:
\begin{itemize}
\item the torsion quotient $\sQ_p(f)$ of ${\mathcal O}^{\oplus r}_X$ of degree
$d_p$, and

\item the torsion quotient $\sE(f)^*/\sF^*$ of $\sE(f)^*$ of degree $d_z$.
\end{itemize}
(It should be clarified that ``knowing the torsion quotient $\sQ_p(f)$'' means
knowing the sheaf $\sQ_p(f)$ along with the surjective homomorphism
${\mathcal O}^{\oplus r}_X\,\longrightarrow\, \sQ_p(f)$; similarly
``knowing the torsion quotient $\sE(f)^*/\sF^*$'' means knowing the sheaf $\sE(f)^*/\sF^*$
along with the surjective homomorphism from $\sE(f)^*$ to it.)
Indeed, once we know $\sQ_p(f)$, we know the kernel $\sE(f)$ and hence know $\sE(f)^*$; if
we know the quotient $\sE(f)^*/\sF^*$, then we know the subsheaf $\sF^*$ of $\sE(f)^*$.
The dual of this inclusion $\sF^*\, \hookrightarrow\, \sE(f)^*$, namely the homomorphism
$$
\sE(f)\,\longrightarrow\, \sF\, ,
$$
gives the meromorphic homomorphism $f$. In other words, we have the diagram
$$
\begin{matrix}
&& 0 &&&&&&\\
&& \Big\downarrow &&&&&&\\
0& \longrightarrow & {\mathcal E}(f) & \longrightarrow & {\mathcal O}^{\oplus r}&
\longrightarrow & {\mathcal Q_p}& \longrightarrow & 0\\
&& ~\Big\downarrow \widehat{f}&&&&&&\\
&& {\mathcal F} &&&&&&\\
&& \Big\downarrow &&&&&&\\
&& {\mathcal Q_z}&&&&&&\\
&& \Big\downarrow &&&&&&\\
&& 0 &&&&&&
\end{matrix}
$$

Let ${\quot}(r,d_p)$ be the Quot scheme parametrizing the torsion quotients
of ${\mathcal O}^{\oplus r}_X$ of degree $d_p$. We have the tautological short exact
sequence of coherent analytic sheaves on $X\times {\mathcal Q}(r,d_p)$
\begin{equation}\label{e-2}
0 \, \longrightarrow \, {\mathcal K}^\univ \,\longrightarrow \, p^*_X{\mathcal O}^{\oplus r}_X
\,\longrightarrow \, \sQ^\univ\,\longrightarrow \, 0\, ,
\end{equation}
where $p_X$ is the projection of $X\times {\quot}(r,d_p)$ to $X$. We write $\sK=\sK^\univ$.
 Now consider
the dual vector bundle
$$
{{\mathcal K}}^*\, \longrightarrow\, X\times {\quot}(r,d_p)\,
\stackrel{p_Q}{\longrightarrow}\, {\quot}(r,d_p)\, ,
$$
where $p_Q$ is the natural projection.
Using $p_Q$, we will consider $\sK^*$ as a family of vector bundles on $X$
parametrized by ${\quot}(r,d_p)$. For any point $y\, \in\, {\quot}(r,d_p)$, the
vector bundle $\sK^*\vert_{X\times\{y\}}$ over $X$ will be denoted by
$\sK^*_{\vert y}$. Let
\begin{equation}\label{e5}
\varphi\, :\, {\quot}(r,d_p,d_z)\, :=\, {\quot}
(\sK^*/\quot(r,d_p), d_z)\, \longrightarrow\, {\mathcal Q}(r,d_p)
\end{equation}
be the relative Quot scheme over ${\quot}(r,d_p)$,
for the family $\sK^*$, parametrizing the torsion
quotients of degree $d_z$. Therefore, for any point $y\,\in\, {\quot}(r,d_p)$,
the fiber $\varphi^{-1}(y)$ is the Quot scheme parametrizing the torsion quotients
of degree $d_z$ of the vector bundle ${\mathcal K}^*_{\vert y}$.

Both ${\quot}(r,d_p)$ and the fibers of $\varphi$ are irreducible smooth projective
varieties. The morphism $\varphi$ is smooth.
Therefore, the projective variety ${\quot}(r,d_p,d_z)$ is
irreducible and smooth.

Consider ${\rQ}^0$ defined in \eqref{q0}. We have a map
$$
\eta'\, :\, {\rQ}^0\,\longrightarrow\, {\quot}(r,d_p)
$$
that sends any triple $(\sF\, ,U\, , f)\, \in\, {\rQ}_0$ to the point representing
the quotient $Q_p(f)$ in \eqref{e3}. Let
\begin{equation}\label{e6}
\eta\, :\, {\rQ}^0\,\longrightarrow\, {\quot}(\sK^*, d_z)\,=\,
{\quot}(r,d_p,d_z)
\end{equation}
be the map that sends any point $\alpha\, =\, (\sF\, ,U\, , f)\, \in\, {\rQ}_0$
to the point of $\varphi^{-1}(\eta'(\alpha))$ that represents the quotient
$\sE(f)^*/\sF^*$ in \eqref{e-1}. This map $\eta$ is injective because, as observed earlier,
the equivalence class of the pair $(\sF\, ,f)$ can be recovered from the quotient $\sQ_p(f)$ of
${\mathcal O}^{\oplus r}_X$ and the quotient $\sE(f)^*/\sF^*$ of $\sE(f)^*$.
The image of $\eta$ is clearly a Zariski open subset of $\quot(r,d_p,d_z)$.

Let $\bigwedge^r {\mathcal K}\, \longrightarrow\,\bigwedge^r p^*_X{\mathcal O}^{\oplus r}_X
\,=\, p^*_X{\mathcal O}_X$
be the $r$-th exterior power of the homomorphism in \eqref{e-2}. Considering it as a
family of subsheaves of ${\mathcal O}_X$ of degree $-d_p$ parametrized by ${\quot}(r,d_p)$, 
we have the corresponding classifying morphism
$$
\delta'_1\, :\, {\quot}(r,d_p)\, \longrightarrow\, {\quot}(1, d_p)
\,=\, \text{Sym}^{d_p}(X)\, .
$$
Let
\begin{equation}\label{e8}
\delta_1\,:=\, \delta'_1\circ\varphi\, :\, {\quot}(r,d_p,d_z)
\,=:\, Q\,\longrightarrow\, \text{Sym}^{d_p}(X)
\end{equation}
be the composition, where $\varphi$ is constructed in \eqref{e5}. Next, consider the
tautological subsheaf
$$
{\mathcal S}\, \hookrightarrow\, (\text{Id}_X\times \varphi)^*{\mathcal K}^*
$$
on $X\times \rQ$. Let $\bigwedge^r {\mathcal S}\, \hookrightarrow\,\bigwedge^r
(\text{Id}_X\times \varphi)^*\sK^*$ be the $r$-th exterior power of the above
inclusion. Let
\begin{equation}\label{e9}
\delta_2\,:\, {\rQ}\,\longrightarrow\, \text{Sym}^{d_z}(X)
\end{equation}
be the morphism that sends any $y\, \in\, \rQ$ to the scheme theoretic
support of the quotient sheaf
$$
(\bigwedge\nolimits^r(\text{Id}_X\times \varphi)^*{\sK}^*_{\vert \varphi(y)})/
(\bigwedge\nolimits^r {\mathcal S}\vert_{X\times\{y\}})\, \longrightarrow\, X\, .
$$
Now define the morphism
\begin{equation}\label{e10}
\delta\,:=\, (\delta_1\, ,\delta_2)\,:\, {\quot}(r,d_p,d_q)\,=\,
{\rQ}\,\longrightarrow\, \text{Sym}^{d_p}(X)\times\text{Sym}^{d_z}(X)\, ,
\end{equation}
where $\delta_1$ and $\delta_2$ are constructed in \eqref{e8} and \eqref{e9}
respectively. It can be shown that $\delta$ is surjective. In fact, in Section
\ref{se3} we will construct, and use, a section of $\delta$.

\begin{remark}
Let ${\mathcal M}_X(r, d_z-d_p)$ denote the moduli stack of vector bundles on $X$
of rank $r$ and degree $d_z-d_p$. Since there is a universal bundle over
$X\times {\quot}(r,d_p,d_q)$, we get a morphism
$$
{\quot}(r,d_p,d_q)\,\longrightarrow\, {\mathcal M}_X(r, d_z-d_p)\, .
$$
\end{remark}

\section{Fundamental group of ${\quot}(r,d_p,d_z)$}\label{se3}

\begin{proposition}\label{prop1}
The homomorphism between fundamental groups induced by the morphism $\delta$ in
\eqref{e10} is an isomorphism.
\end{proposition}

\begin{proof}
We will first construct a section of $\delta$. Let
$$
D(d_p)\, \subset\, X\times \text{Sym}^{d_p}(X)
$$
be the divisor consisting of all $(x\, , \{y_1\, ,\cdots\, , y_{_{d_p}}\})$ such that
$x\, \in\, \{y_1\, ,\cdots\, , y_{_{d_p}}\}$. Then the subsheaf
$$
{\mathcal O}_{X\times \text{Sym}^{d_p}(X)}(-D(d_p))\oplus
{\mathcal O}^{\oplus r-1}_{X\times \text{Sym}^{d_p}(X)}\, \subset\,
{\mathcal O}^{\oplus r}_{X\times \text{Sym}^{d_p}(X)}
$$
produces a classifying morphism
\begin{equation}\label{t1}
\theta_1\, :\, \text{Sym}^{d_p}(X)\, \longrightarrow\, {\quot}(r,d_p)\, .
\end{equation}
Let $\xi_1$ (respectively, $\xi_2$) denote the projection of
$\text{Sym}^{d_p}(X)\times\text{Sym}^{d_z}(X)$ to $\text{Sym}^{d_p}(X)$ (respectively,
$\text{Sym}^{d_z}(X)$). Like before, $D(d_z)\, \subset\, X\times \text{Sym}^{d_p}(X)$
be the divisor consisting of all $(x\, , \{y_1\, ,\cdots\, , y_{_{d_z}}\})$ such that
$x\, \in\, \{y_1\, ,\cdots\, , y_{_{d_z}}\}$. The subsheaf
$$
((\text{Id}_X\times\xi_1)^*({\mathcal O}_{X\times \text{Sym}^{d_p}(X)}(D(d_p)))\otimes
(\text{Id}_X\times\xi_2)^*({\mathcal O}_{X\times \text{Sym}^{d_z}(X)}(-D(d_z))))
$$
$$
\oplus ({\mathcal O}^{\oplus r-1}_{X\times \text{Sym}^{d_p}(X)\text{Sym}^{d_z}(X)})^*
\, \subset\,
(\text{Id}_X\times\xi_1)^*({\mathcal O}_{X\times \text{Sym}^{d_p}(X)}(-D(d_p))
\oplus {\mathcal O}^{\oplus r-1}_{X\times \text{Sym}^{d_p}(X)})^*
$$
produces a classifying morphism
\begin{equation}\label{theta}
\theta\, :\, \text{Sym}^{d_p}(X)\times\text{Sym}^{d_z}(X)\,\longrightarrow\,
{\quot}(r,d_p,d_q)\, .
\end{equation}
We note that $\varphi\circ\theta\,=\,\theta_1$, where $\varphi$ and $\theta_1$ are the
morphisms constructed in \eqref{e5} and \eqref{t1} respectively.

It is straightforward to check that
\begin{equation}\label{e11}
\delta\circ\theta\,=\, \text{Id}_{\text{Sym}^{d_p}(X)\times\text{Sym}^{d_z}(X)}\, ,
\end{equation}
where $\delta$ is constructed in \eqref{e10}. In view of this section $\theta$, we
conclude that the induced homomorphism between fundamental groups
$$
\delta_*\, :\, \pi_1({\quot}(r,d_p,d_q))\,\longrightarrow\, \pi_1(
\text{Sym}^{d_p}(X)\times\text{Sym}^{d_z}(X))
$$
is surjective (the base points of fundamental groups are suppressed in the notation).

Let
\begin{equation}\label{u}
U\, \subset\, \text{Sym}^{d_p}(X)\times\text{Sym}^{d_z}(X))
\end{equation}
be the Zariski open subset consisting of all
$$
(x\, ,y)\,=\, (\{x_1\, ,\cdots\, ,x_{_{d_p}}\}\, ,\{y_1\, ,\cdots\, ,y_{_{d_z}}\})
\,\in\, \text{Sym}^{d_p}(X)\times\text{Sym}^{d_z}(X))
$$
such that the $d_p+d_z$ points $\{x_1\, ,\cdots\, ,x_{_{d_p}}\, ,y_1\, ,\cdots\, ,
y_{_{d_z}}\}$ are all distinct, equivalently, the effective divisor $x+y$ is reduced. Let
\begin{equation}\label{ep}
\theta_0\, :=\, \theta\vert_U :\, U\,\longrightarrow\, {\quot}(r,d_p,d_z)
\end{equation}
be the restriction of the map $\theta$ in \eqref{theta}. Also, consider the restriction
\begin{equation}\label{vpp}
\delta_0\, :=\, \delta\vert_{\delta^{-1}(U)}\, :\, \delta^{-1}(U)\,
\longrightarrow\, U\, .
\end{equation}
Every fiber of $\delta_0$ is identified with $({\mathbb P}^{r-1}_{\mathbb C})^{d_p}\times
(P^{r-1}_{\mathbb C})^{d_z}$, where ${\mathbb P}^{r-1}_{\mathbb C}$ is the projective
space parametrizing the hyperplanes in ${\mathbb C}^r$ and $P^{r-1}_{\mathbb C}$ is the
projective space parametrizing the lines in ${\mathbb C}^r$ (so
$P^{r-1}_{\mathbb C}$ parametrizes the hyperplanes in $({\mathbb C}^r)^*$). From the
homotopy exact sequence associated to $\delta_0$ it follows that the induced
homomorphism of fundamental groups
$$
\delta_{0,*}\, :\, \pi_1(\delta^{-1}(U))\,\longrightarrow\, \pi_1(U)
$$
is an isomorphism. The variety $\quot(r,d_p,d_z)$ is smooth, and $\delta^{-1}(U)$ is a
nonempty Zariski open subset of it. Therefore, the homomorphism
$$
\iota_*\, :\, \pi_1(\delta^{-1}(U))\,\longrightarrow\, \pi_1({\quot}(r,d_p,d_z))
$$
induced by the inclusion $\iota\, :\, \varphi^{-1}(U)\,\hookrightarrow\,
\quot (r,d_p,d_z)$ is surjective. Since $\delta_{0,*}$ is an isomorphism, this
implies that
the homomorphism
$$
\theta_{0,*}\, :\, \pi_1(U)\,\longrightarrow\, \pi_1({\quot}(r,d_p,d_z))
$$
induced in $\theta_0$ in \eqref{ep} is surjective. Since $\theta_0$ extends to $\theta$,
this immediately implies that the homomorphism
$$
\theta_*\, :\, \pi_1(\text{Sym}^{d_p}(X)\times\text{Sym}^{d_z}(X))\,
\longrightarrow\, \pi_1({\quot}(r,d_p,d_z)
$$
induced in $\theta$ in \eqref{theta} is surjective. Since $\theta_*$ is surjective, and
the composition $\delta_*\circ\theta_*$ is injective (see \eqref{e11}) we
conclude that $\delta_*$ is also injective.
\end{proof}

\section{Cohomology of $\quot(r,d_p,d_z)$}

\subsection{Generalization of a theorem of Bifet}

Let $S_1\, ,S_2\, ,\cdots\, , S_k$ be a smooth connected projective varieties
over $\mathbb C$. Fix some line bundles $\sL_i$ on $S_i\times X$ of relative degree $d_i$
over $S_i$. In other words
\[\deg (\sL_i\vert_{s\times X}) \,=\, d_i\] for each point $s\,\in\, S_i$.
Set $S\,=\,S_1\times \cdots \times S_k$. Let
$$
\pi_{S_i\times X}\, :\, S\times X \, \longrightarrow\,S_i\times X
$$
be the natural projection. Define
$$
\widetilde{\sL}_i\, :=\, \pi_{S_i\times X}^*\sL_i\, .
$$
Let
$$
\phi \, :\, \quot(\oplus_i \widetilde{\sL}_i/S, d)\, \longrightarrow\, S
$$
be the relative Quot scheme that parametrizes the torsion quotients of degree $d$.
So for any $s\, =\, (s_1\, ,\cdots\ , s_k)\, \in\, S$, the fiber $\phi^{-1}(s)$
parametrizes the torsion quotients of $\oplus_{i=1}^k \sL_i\vert_{s_i\times X}$
of degree $d$. By deformation theory, $\phi$ is a smooth morphism
of relative dimension $kd$, so $\quot(\oplus_i \widetilde{\sL}_i/S, d)$ is smooth
of dimension $\dim(S)+kd$. The torus $\gm^k$ acts on $\quot(\oplus_i
\widetilde{\sL}_i/S, d)$ via its action on $\bigoplus_{i=1}^k \widetilde{\sL}_i$. 

For any positive integer $p$, let $\quot({\sL}_i/S_i,p)\,\longrightarrow
\, S_i$ denote the relative Quot scheme parametrizing the torsion quotients of
${\sL}_i/S_i$ of degree $p$. So the fiber of $\quot({\sL}_i/S_i,p)$ over any
$s\, \in\, S_i$ parametrizes the torsion quotients of ${\sL}_i\vert_{s\times X}$
of degree $p$.

\begin{proposition}
There is a bijection between the partitions $\bP\, =\, (p_1\, ,p_2\, ,\cdots\, , p_k)$
of $d$ of length $k$ and the connected components of the fixed point loci of the
$\gm^k$ action on $\quot(\oplus_i \widetilde{\sL}_i/S, d)$. The component corresponding
to the partition $\sum_{i=1}^k p_i \,=\,d$ is the product of Quot schemes
\[
\quot(\sL_\bullet/S,\bP)\,:=\,\quot({\sL}_1/S_1,p_1)\times
\quot({\sL}_2/S_2, p_2) \times \ldots \times \quot({\sL}_k/S_k,p_k)
\]
with the obvious structure morphism to $S$.
\end{proposition}

\begin{proof}
On applies the argument used to prove Lemme 1 in \cite{Bifet}.
\end{proof}

As all schemes and morphisms are assumed to be projective it is possible to 
choose a one parameter subgroup
\[
\gm\,\hookrightarrow \,\gm^k
\]
so that 
\[
\quot(\oplus_i \widetilde{\sL}_i/S,d)^{\gm} \,=\, \quot(\oplus_i
\widetilde{\sL}_i/S, d)^{\gm^k}\, .
\]
Further, the above one-parameter subgroup can be chosen to be given by an
increasing sequence of weights $\lambda_1\,<\,\lambda_2\,<\,\ldots \,<\, \lambda_k$.

There is an induced action of $\gm$ on the tangent space at a fixed point $x$. 
The action preserves the normal space to the fixed point locus and we wish to
describe the subspace of positive weights.

Take a partition $\bP\,= \,(p_1\, ,\cdots \, ,p_k)$ of $D$. As before, let
$$
\quot(\sL_\bullet/S,\bP) \, \subset\, \quot(\oplus_i
\widetilde{\sL}_i/S,d)^{\gm^k}
$$
be the connected component corresponding to $\bP$. For a point $x\, \in\,
\quot(\sL_\bullet/S,\bP)$, its image in $S_i$ will be denoted by $x_i$. The line
bundle $\sL_i\vert_{x_i\times X}$ on $X$ will be denoted by $\sL^x_i$.
The point $x_i$ is given by the exact sequence
\[
0\,\longrightarrow\, \sL^x_i\otimes {\mathcal O}_X(-D_i)\,
\longrightarrow\, \sL^x_i\,\longrightarrow\,\sO_{D_i}
\,\longrightarrow\, 0
\]
where $D_i$ is an effective divisor on $X$ with $\deg D_i\,= \,p_i$.
The relative tangent bundle for the projection $\phi$ is
$$
T_x \quot(\oplus_i \widetilde{\sL}_i/S,d)/S\,=\, 
\bigoplus_{i,j=1}^k \Hom(\sL^x_i\otimes {\mathcal O}_X(-D_i)\, ,\sO_{D_j})\, .
$$
On the other hand, the relative tangent space to the fixed point locus
$\quot(\sL_\bullet/S,\bP)$ is
\[
T_x \quot(\sL_\bullet/S,\bP)/S \,=\, \bigoplus_{i=1}^k
\Hom(\sL^x_i\otimes {\mathcal O}_X(-D_i)\, ,\sO_{D_i})\, .
\]
Consequently, the normal bundle $N$ to $\quot(\sL_\bullet/S,\bP)\,\subset\,
\quot(\oplus_i \widetilde{\sL}_i/S,d)$ is
\[
N_x \,=\,\bigoplus_{i\ne j} \Hom(\sL^x_i\otimes {\mathcal O}_X(-D_i\, , \sO_{D_j})\, .
\]
Also, the subspace of positive weights is 
\[
N^+_x \,=\, \bigoplus_{i< j} \Hom(\sL^x_i\otimes {\mathcal O}_X(-D_i\, ,\sO_{D_j})
\]
because the torus acts on $\Hom(\sL^x_i\otimes {\mathcal O}_X(-D_i\, ,\sO_{D_j})$
with weight $\lambda_j-\lambda_i$. It follows that 
\[
d(\bP)\,:=\, \dim N^+_x \,= \,\sum_{i=1}^k (i-1)p_i\, . 
\]

\begin{proposition}\label{p:bifet}
The Quot schemes for line bundles $$\quot(\sL_i/S_i, p_i)\,=\,
\quot(\sO/S_i,p_i)\,=\,\sym^{p_i}(X)\times S_i\, .$$
The Poincar\'e polynomial of $\quot(\oplus_i\widetilde{\sL}_i/S, d)$ is given by
\begin{eqnarray*}
P(\quot(\oplus_i\widetilde{\sL}_i/S,d),t) & = & \sum_{\bP} t^{2d(\bP)}
P(\quot(\sL_i/S_i,p_i) ,t) \\
& = & \sum_{\bP} t^{2d(\bP)}P(S_i,t)P(\sym^{p_i}(X),t)\, ,
\end{eqnarray*}
where the sum is over all partitions of $d$ of length $k$.
\end{proposition}

\begin{proof} 
The isomorphism $ \quot(\sO/S_i, p_i)\,\stackrel{\sim}{\longrightarrow}\,
\quot(\sL_i/S_i,p_i)$ is by tensoring exact sequences with
$\sL_i$. The second equality is via (\ref{l:qPullback}).

We need to recall the theorems of \cite{bb:73} and \cite{kirwan:88} in our present
context. The torus action determines two stratifications of the variety
$\quot(\oplus_i\widetilde{\sL}_i/S,d)$. The strata are in bijection with connected
components of the fixed point locus which are in turn in bijection with partitions of
$d$ of length $k$. Given such a partition $\bP$, its corresponding strata are 
\[
\quot(\sL_\bullet/S,\bP)^+ \,:=\,
\{ x\,\mid\, \lim_{t\to 0}t.x \,\in\, \quot(\sL_\bullet/S,\bP)\}
\]
and 
\[
\quot(\sL_\bullet/S,\bP)^- \,:=\, \{ x\,\mid\, \lim_{t\to \infty}t.x \,\in
\, \quot(\sL_\bullet/S,\bP)\}\, .
\]
Both of these stratifications are known to be perfect. There are affine fibrations 
\[
\quot(\sL_\bullet/S,\bP)^+ \,\longrightarrow\, \quot(\sL_\bullet/S,\bP)\quad\text{and}\quad
\quot(\sL_\bullet/S,\bP)^-\,\longrightarrow\, \quot(\sL_\bullet/S,\bP)
\]
of relative dimensions $\dim N^+_x$ and $\dim N^-_x$ respectively, where
$x\,\in\,\quot(\sL_\bullet/S,\bP)$ is an arbitrary closed point. It follows that the
codimension of $\quot(\sL_\bullet/S,\bP)^-$ is $\dim N^+_x$ which gives the above
formula for the Poincare polynomial. 
\end{proof}

\subsection{The cohomology of $\quot(r,d_p,d_z)$}

In this subsection we describe the Poincar\'e polynomial of $\quot(r,d_p,d_z)$.
Consider the morphism $\varphi$ in \eqref{e5}. There is a natural action of the
torus $\gm^r$ on the target $\quot(\sO^r,d_p)\, =\, {\mathcal Q}(r,d_p)$. This
action clearly lifts to the domain $\quot(r,d_p,d_z)$ for $\varphi$.

The previous subsection provides us with a decomposition and an induced
formula for the Poincar\'e polynomial of $\quot(r,d_p,d_z)$. Let us recall it
quickly in the present context. There is a bijection between connected components of
fixed point locus and partitions of $d_p$ of length $r$. Given a partition
$\bP\,=\, (p_1\, ,p_2\, ,\cdots\, , p_r)$, the corresponding
component of $\quot(\sO^r,d_p)^{\gm}$ is
\begin{eqnarray*}
\quot(\sO,p_1)\times\cdots \times \quot(\sO,p_r) &=& 
\sym^{p_1}(X) \times \cdots \times \sym^{p_r}(X) \\
&=& \sym^\bP X.
\end{eqnarray*}
There are universal divisors $D_{p_i}^\univ$ inside $\sym^{p_i}(X)\times X$. The
component of $\quot(r,d_p,d_z)^{\gm^r}$ corresponding to $\bP$, that is 
\[
\phi^{-1}(\sym^{p_1}(X) \times \sym^{p_2}(X)\times \cdots\times \sym^{p_r}(X))
\]
is then identified with $\quot(\oplus_i\sO_{\sym^{p_i}(X)\times X}(D_{p_i}^\univ)
/\sym^{p_i}(X),d_z)$. As the morphism $\varphi$ in \eqref{e5} is smooth, and smooth
morphisms preserve codimension, we obtain the following formula for the Poincar\'e
polynomial:
 \begin{equation}\label{e:poincare1}
 P(\quot(r,d_p,d_z),t) \,=\, \sum_{\bP} t^{2d(\bP)}P(\quot(\oplus_i
\sO_{\sym^{p_i}(X)\times X}(D_{p_i}^\univ)/\sym^{p_i}(X),d_z),t).
 \end{equation}
To complete the calculation we need to compute the Poincar\'e polynomials of
$$\quot(\oplus_i\sO_{\sym^{p_i}(X)\times X}(D_{p_i}^\univ)/\sym^{\bP}(X),d_z)\, .$$

Once again Proposition \ref{p:bifet} applies. The connected components of the fixed
point loci are in bijection with partitions of $d_z$ of length $r$. Given a partition
$\bQ\,=\,(q_1\, ,\cdots\, , q_r)$, the corresponding connected component is
\[
\quot(\sO_{\sym^{p_1}(X)\times X}(-D_{p_1})/\sym^{p_1}(X),q_1)\times \cdots \times
\quot(\sO_{\sym^{p_r}(X)\times X}(-D_{p_r})/\sym^{p_r}(X),q_r)
\]
which is canonically isomorphic to
\[
\sym^{\bP,\bQ}X\,:=\,\sym^{p_1}(X)\times\cdots\times\sym^{p_r}(X)\times \sym^{q_1}(X)
\times\cdots\times\sym^{q_r}(X)\, .
\]
We obtain the following formula:
$$
P(\quot(\oplus_i\sO{\sym^{p_i}(X)\times X}(D_{p_i}^\univ)/\sym^{\bP}(X),d_z))
\,=\, \sum_{\bQ} t^{2d(\bQ)}P(\sym^{\bP,\bQ}(X),t)\, .
$$

Putting this all together we obtain the following:

\begin{theorem}\label{thm1}
The Poincar\'e polynomial for $\quot(r,d_p,d_z)$ is
\[
P(\quot(r,d_p,d_z),t) \,=\, \sum_{\bP}\sum_{\bQ} t^{2[d(\bP)+d(\bQ)]}P(\sym^{\bP}(X),
t)P(\sym^\bQ (X),t)\, ,
\]
where $\bP$ varies over all partitions of $d_p$ of length $r$ and $\bQ$ varies
over all partitions of $d_z$ of length $r$.
\end{theorem}

Poincar\'e polynomial of $\sym^{n}(X)$ is the coefficient of $t^n$ in
$$
\frac{(1+tx)^{2g_X}}{(1-t)(1-tx^2)}\, ,
$$
where $g_X$ is the genus of $X$ \cite[p. 322, (4.3)]{Ma}. Using this and
Theorem \ref{thm1} we get an explicit expression for $P(\quot(r,d_p,d_z),t)$.

\section*{Acknowledgements}

We thank the National University of Singapore for its hospitality.



\begin{thebibliography}{AAAAA}

\bibitem[Ba]{Ba} J.M. Baptista, On the $L^2$-metric of vortex moduli spaces,
\textit{Nucl. Phys. B} \textbf{844} (2011), 308--333.

\bibitem[BDW]{BDW} A. Bertram, G. Daskalopoulos and R. Wentworth, Gromov invariants
for holomorphic maps from Riemann surfaces to Grassmannians,
\textit{Jour. Amer. Math. Soc.} \textbf{9} (1996) 529--571.

\bibitem[Bia]{bb:73} A. Bia{\l}ynicki-Birula, Some theorems on actions of 
algebraic groups, \textit{Ann. of Math.} \textbf{98} (1973), 480--497.

\bibitem[Bif]{Bifet} E. Bifet, Sur les points fixes sch\'ema
${\rm Quot}_{{\mathcal O}_X/X,k}$ sous l'action du tore ${\mathbf G}^r_{m,k}$,
\textit{Com. Ren. Math. Acad. Sci. Paris} \textbf{309} (1989), 609--612.

\bibitem[BR]{BR} I. Biswas and N. M. Rom\~ao, Moduli of vortices and Grassmann
manifolds, {\it Comm. Math. Phy.} \textbf{320} (2013), 1--20.

\bibitem[Br]{Br} S. Bradlow, Vortices in holomorphic line
bundles over closed K\"ahler manifolds, \textit{Commun. Math. Phys.}
\textbf{135} (1990), 1--17.

\bibitem[EINOS]{EINOS} M. Eto, Y. Isozumi, M. Nitta, K. Ohashi and N. Sakai,
Solitons in the Higgs phase: the moduli matrix approach,
\textit{J. Phys. A: Math. Gen.} \textbf{39} (2006), 315--392.

\bibitem[Gr]{Gr} A. Grothendieck, \textit{Techniques de construction et
th\'eor\`emes d'existence en g\'eom\'etrie alg\'ebrique. {IV}. {L}es sch\'emas de
{H}ilbert}, {S\'eminaire {B}ourbaki, {V}ol.\ 6}, {Exp.\ No.\ 221, 249--276}.

\bibitem[Ki]{kirwan:88} F. Kirwan, Intersection homology and torus actions,
\textit{Jour. Amer. Math. Soc.} \textbf{1} (1988), 385--400.

\bibitem[Ma]{Ma} I. G. Macdonald, Symmetric products of an algebraic curve,
\textit{Topology} \textbf{1} (1962), 319--343.

\end{thebibliography}
\end{document}